\documentclass[a4paper,12pt]{article}
\usepackage[top=2.5cm,bottom=2.5cm,left=2.5cm,right=2.5cm]{geometry}
\usepackage{amsmath,amssymb,amsthm,graphics,latexsym,amsfonts}
\usepackage{makecell}
\usepackage{caption}
\usepackage{booktabs}
\usepackage{graphicx}
\usepackage{bm,bbm}
\usepackage{mathrsfs}
\usepackage{tabularx}
\usepackage{cite}
\usepackage{indentfirst}

\newtheorem{thm}{Theorem}[section]
\newtheorem{lem}[thm]{Lemma}
\newtheorem{cor}[thm]{Corollary}

\newtheorem{eg}[thm]{Example}

\makeatletter
\newcommand{\Rmnum}[1]{\expandafter\@slowromancap\romannumeral #1@} %罗马大写字母
\makeatother
\allowdisplaybreaks[3] %公式可以跨行 用到\usepackage{amsmath}

\begin{document}

\title{{Matching anti-forcing polynomials of catacondensed hexagonal systems}}
\author{Shuang Zhao$^{1,2}$\\
{\small $^{1}$School of Information Engineering, Lanzhou University of Finance and Economics, }\\
{\small Lanzhou, Gansu 730000, P. R. China;}\\
{\small E-mail: zhaosh2018@126.com}}
\date{}
\maketitle

\begin{abstract}
    In this paper, we derive a recurrence relation of anti-forcing polynomial for catacondensed hexagonal systems.
\end{abstract}

\section{Introduction}
    Let $G$ be a graph with vertex set $V(G)$ and edge set $E(G)$. A \emph{perfect matching} $M$ of $G$ is a set of disjoint edges that covers all vertices of $G$. Lei, Yeh and Zhang \cite{Lei} generalized \emph{anti-forcing number} to single perfect matching, and Klein and Rosenfeld \cite{new} presented the same concept as \emph{(e)-forcing}. An \emph{anti-forcing set} $S$ of $M$ is a subset of $E(G)\backslash M$ such that $G-S$ has a unique perfect matching $M$, where $G-S$ denotes the subgraph obtained from $G$ by deleting the edges that belong to $S$. The \emph{anti-forcing number} of $M$ is the minimum cardinality over all anti-forcing sets of $M$, denoted by $af(G,M)$. The \emph{maximum} (resp. \emph{minimum}) \emph{anti-forcing number} of $G$ is the maximum (resp. minimum) value of $af(G,M)$ over all perfect matchings $M$ of $G$, denoted by $Af(G)$ (resp. $af(G)$). The set of anti-forcing numbers of all perfect matchings of $G$ is called the \emph{anti-forcing spectrum} of $G$.

    A \emph{hexagonal system} is a 2-connected finite plane graph such that every interior face is surrounded by a regular hexagon of side length one. Lei et al. \cite{Lei} showed that the maximum anti-forcing number of a hexagonal system equals the \emph{Fries number}, namely the maximum number of alternating hexagons with respect to a perfect matching, which can measure the stability of a benzenoid hydrocarbon.

    Hwang et al. \cite{lei2} introduced the \emph{anti-forcing polynomial} of a graph $G$ as
    \begin{equation}
        \label{equ1}
        Af(G,x)=\sum_{M\in\mathcal{M}(G)}{{x}^{af(G,M)}}=\sum _{ i=af(G) }^{ Af(G) }{ \upsilon(G,i){ x }^{ i } },
    \end{equation}
    where $\mathcal{M}(G)$ denotes the set of all perfect matchings of $G$, $\upsilon(G,i)$ denotes the number of perfect matchings of $G$ with anti-forcing number $i$. By the definition, we can derive the following lemma, which shows that a disconnected graph can be partitioned into components to consider.

    \begin{lem}
        Let $G$ be a graph with the components $G_1,G_2,\ldots,G_k.$ Then
        \begin{align*}
            Af(G,x)=\prod_{i=1}^{k}{Af({G}_{i},x)}.
        \end{align*}
    \end{lem}

    From the above lemma, we know that if the anti-forcing spectrum of $G_i$ is an integer interval $[a_i,b_i]$ for $i=1,2,\ldots,k$, then the anti-forcing spectrum of $G$ is $[\sum_{i=1}^{k}a_i,\sum_{i=1}^{k}b_i]$.

\section{Some preliminaries}
    We now introduce some properties of anti-forcing set and anti-forcing number. Let $G$ be a graph with a perfect matching $M$. A cycle of $G$ is called \emph{$M$-alternating} if its edges appear alternately in $M$ and $E(G)\setminus M$.

    \begin{thm}\em\cite{Lei}
        \label{anti-forc-equi-defi}
        Let $G$ be a graph with a perfect matching $M.$ A subset $S\subseteq E(G)\backslash M$ is an anti-forcing set of $M$ if and only if each $M$-alternating cycle of $G$ contains at least one edge of $S.$
    \end{thm}

    Two $M$-alternating cycles is called \emph{compatible} if they either are disjoint or intersect only at edges in $M$. A set of $M$-alternating cycles of $G$ is called a \emph{compatible $M$-alternating set} if every two members in it are compatible. It is easy to see from the above theorem that the anti-forcing number $af(G,M)$ is bounded below by $c'(G,M)$, the maximum cardinality over all compatible $M$-alternating sets of $G$. Furthermore, Lei et al. \cite{Lei} proved that all planar bipartite graphs satisfy the lower bound with equality, especially for hexagonal systems.

    \begin{thm}\em\cite{Lei}
        \label{anti-forc-mini-max}
        Let $G$ be a planar bipartite graph$.$ Then for each perfect matching $M$ of $G,$ we have $af(G,M)=c'(G,M).$
    \end{thm}

    For $S\subseteq E(G)\setminus M$, an edge $e$ of $G-S$ is said to be \emph{forced} by $S$ if it belongs to all perfect matchings of $G-S$, thus belongs to $M$. An edge $g$ of $G-S$ is said to be \emph{anti-forced} by $S$ if it belongs to no perfect matchings of $G-S$, thus does not belong to $M$. Let $G\circleddash S$ denote the subgraph obtained from $G-S$ by deleting the ends of all the edges that forced by $S$ and deleting all the edges that anti-forced by $S$. Obviously we have the following lemma.

    \begin{lem}
        \label{anti-forc-pend-edge-judg}
        Let $G$ be a graph with a perfect matching $M.$ A subset $S\subseteq E(G)\setminus M$ is an anti-forcing set of $M$ if and only if $G\circleddash S$ is empty$.$
    \end{lem}

    \begin{lem}
        \label{anti-forc-subs-calc}
        Let $G$ be a graph with a perfect matching $M,$ and $\mathcal{C}$ be a compatible $M$-alternating set of $G.$ Given a subset $S\subseteq E(G)\setminus M,$ which consists of precisely one edge from each cycle in $\mathcal{C}.$ If $S$ anti-forces all the other edges in $E(\mathcal{C})\cap (E(G)\setminus M),$ then
        \begin{align*}
            af(G,M)=af(G\circleddash S,M\cap E(G\circleddash S))+|S|.
        \end{align*}
    \end{lem}
    \begin{proof}
        Let $S'$ be a minimum anti-forcing set of $M\cap E(G\circleddash S)$ in $G\circleddash S$. Then $af(G\circleddash S,M\cap E(G\circleddash S))=|S'|$. Since $G\circleddash (S\cup S')=(G\circleddash S)\circleddash S'$ is empty, $S\cup S'$ is an anti-forcing set of $M$ in $G$ by Lemma \ref{anti-forc-pend-edge-judg}. Suppose $S_0$ is another anti-forcing set of $M$ in $G$ such that $|S_0|<|S\cup S'|$. Then either $|S_0\cap E(G\circleddash S)|<|S'|$ or $|S_0\cap (E(G)\setminus E(G\circleddash S))|<|S|=|\mathcal{C}|$. By assumption, we have $E(\mathcal{C})\cap (E(G)\setminus M)\subseteq E(G)\setminus E(G\circleddash S)$. It follows that either $S_0\cap E(G\circleddash S)$ is not an anti-forcing set of $M\cap E(G\circleddash S)$ in $G\circleddash S$, or there exists an $M$-alternating cycle in $\mathcal{C}$ containing no edges of $S_0\cap (E(G)\setminus E(G\circleddash S))$. This implies that there is an $M$-alternating cycle in $G$ containing no edges of $S_0$, a contradiction to Theorem \ref{anti-forc-equi-defi}. Hence $S\cup S'$ is a minimum anti-forcing set of $M$ in $G$, which implies $af(G,M)=|S|+|S'|$.
    \end{proof}

    By the above lemma, from a special compatible $M$-alternating set we can find a subset that is contained in some minimum anti-forcing set of a perfect matching $M$. In particular, by Theorem \ref{anti-forc-mini-max} for hexagonal systems we can find a minimum anti-forcing set directly.

\section{Catacondensed hexagonal system}
    A hexagonal system $H$ is called \emph{catacondensed} if no three of its hexagons share a common vertex. If $H$ contains at least two hexagons, then every hexagon of $H$ has one, two, or three neighboring hexagons. A hexagon is called \emph{terminal} and \emph{branched} if it has one and three neighboring hexagons respectively. A hexagon with precisely two neighboring hexagons is called a \emph{kink} if it possesses two adjacent vertices with degree two, and is called \emph{linear} otherwise. If $H$ contains precisely one hexagon, then we call it a \emph{terminal hexagon}. In Fig. \ref{cata-eg}(a), we illustrate an example of catacondensed hexagonal system, where $T$ denotes terminal hexagon, $B$ denotes branched hexagon, $K$ denotes kink, and $L$ denotes linear hexagon.

    \begin{figure}[htbp]
        \centering
        \includegraphics[height=1.8in]{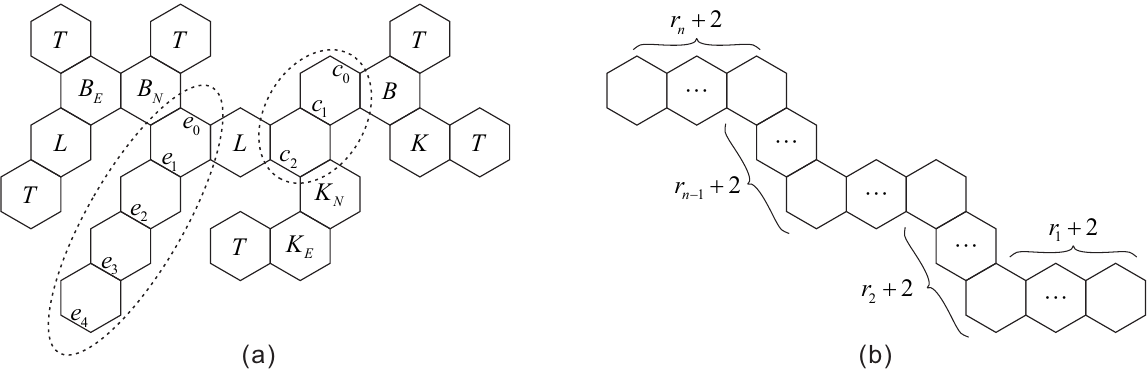}
        \caption{(a) A catacondensed hexagonal system and (b) a hexagonal chain.}
        \label{cata-eg}
    \end{figure}

    A catacondensed hexagonal system with no branched hexagons is called a \emph{hexagonal chain} (see Fig. \ref{cata-eg}(b)). A hexagonal chain with no kinks is called \emph{linear}, and furthermore a linear hexagonal chain $W$ in $H$ is called \emph{maximal} if it is contained in no other linear ones of $H$. Define the \emph{length} of $W$ as the number of hexagons in $W$, and a \emph{vertical edge} of $W$ as an edge that intersects with a straight line passing through all the centers of hexagons in $W$. In Fig. \ref{cata-eg}(a), we illustrate two maximal linear hexagonal chains with length four and two respectively, where $e_i$ is a vertical edge of the first one and $c_j$ is a vertical edge of the second one for $i=0,1,2,3,4$ and $j=0,1,2$.

    \begin{lem}\em\cite{deng2}
        \label{ding}
        Let $H$ be a catacondensed hexagonal system with a perfect matching $M.$ Then every maximal linear hexagonal chain of $H$ has precisely one vertical edge that belongs to $M.$
    \end{lem}

    Before discussing anti-forcing polynomial of catacondensed hexagonal systems $H$, we give some notations first. A kink of $H$ is called an \emph{end kink} if $H$ contains a maximal linear hexagonal chain starting from the kink and ending at a terminal hexagon. In Fig. \ref{cata-eg}(a) we illustrate an end kink as $K_E$, and a kink but not an end one as $K_N$. A branched hexagon of $H$ is called an \emph{end branched hexagon} if $H$ contains two maximal linear hexagonal chains starting from the branched hexagon and ending at terminal hexagons. In Fig. \ref{cata-eg}(a) we illustrate an end branched hexagon as $B_E$, and a branched hexagon but not an end one as $B_N$. Except for linear hexagonal chains, every catacondensed hexagonal system contains a kink or a branched hexagon, thus contains an end kink or an end branched hexagon.

    Define a \emph{tail} $T(t_1,t_2,t_3)$ with respect to $h$ as a subgraph $T_1\cup T_2\cup T_3$ of $H$, where $h$ is a terminal hexagon of $H$ if $H$ is a linear hexagonal chain, and an end kink or an end branched hexagon of $H$ otherwise. In detail, if $H$ is a linear hexagonal chain, then define $T_1=H$ and $T_2=T_3=h$, and define $t_1+2$ ($t_1\geqslant -1$) as the length of $T_1$ and $t_2=t_3=-1$ (see Fig. \ref{cata-defi}(a)). If $h$ is an end kink, then define $T_2=h$, and $T_1,T_3$ as maximal linear hexagonal chains starting from $h$ satisfying that $T_1$ ends at a terminal hexagon, and define $t_2=-1$, $t_1+2,t_3+2$ ($t_1,t_3\geqslant 0$) as the length of $T_1,T_3$ respectively (see Fig. \ref{cata-defi}(b)). If $h$ is an end branched hexagon, then define $T_1,T_2,T_3$ as maximal linear hexagonal chains starting from $h$ satisfying that $T_1,T_2$ end at terminal hexagons, and define $t_1+2,t_2+2,t_3+2$ ($t_1,t_2,t_3\geqslant 0$) as the length of $T_1,T_2,T_3$ respectively (see Fig. \ref{cata-defi}(c)).

    \begin{figure}[htbp]
        \centering
        \includegraphics[height=4.2in]{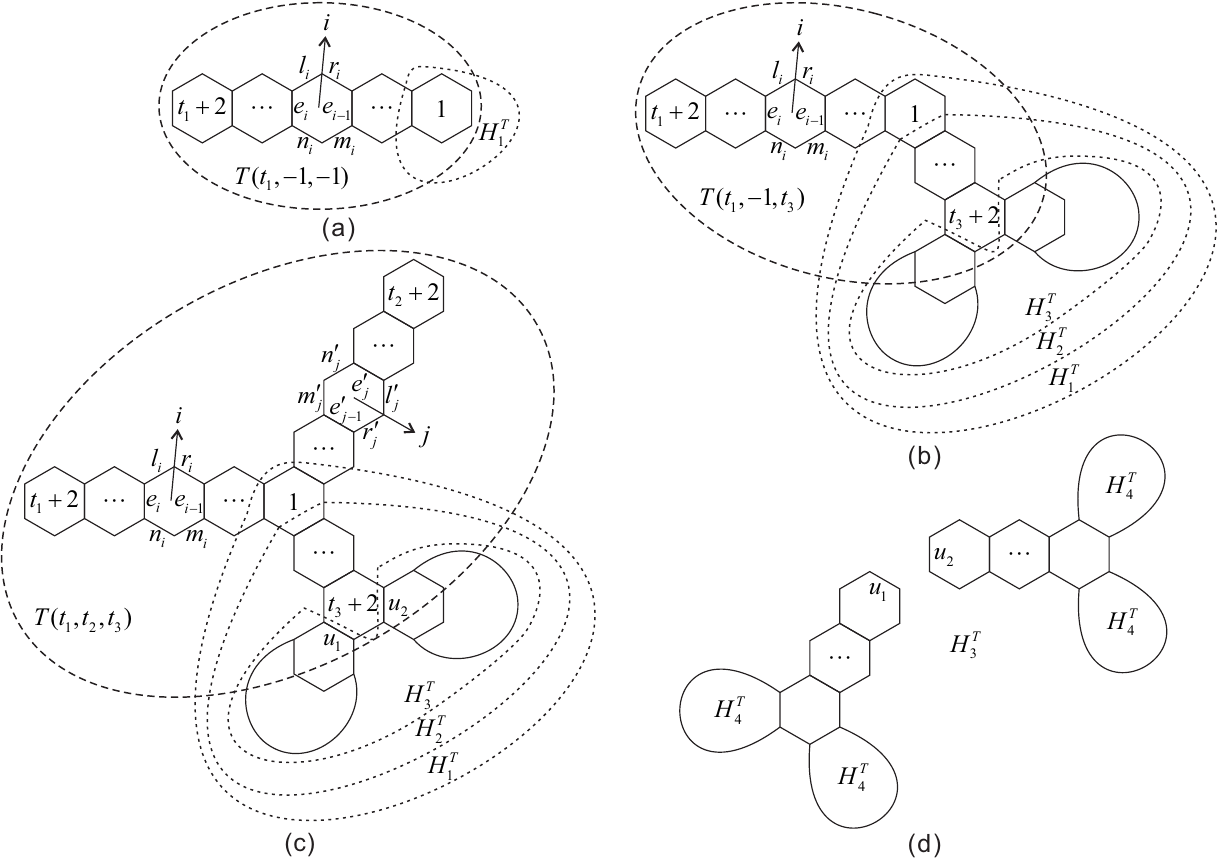}
        \caption{Illustration of tail $T(t_1,t_2,t_3)$ and $H^T_i$ for $i=1,2,3,4$.}
        \label{cata-defi}
    \end{figure}

    For a tail $T(t_1,t_2,t_3)$ with respect to $h$, denote $H^T_1$ the subgraph of $H$ by deleting all the hexagons in $T_1\cup T_2$ except for $h$, denote $H^T_2$ the subgraph of $H$ by deleting all the hexagons in $T_1\cup T_2$, and denote $H^T_3$ the subgraph of $H$ by deleting all the hexagons in $T(t_1,t_2,t_3)$ (see Figs. \ref{cata-defi}(a-c)). Note that $H^T_1$ is a \emph{smaller} catacondensed hexagonal system than $H$ (namely containing the number of hexagons less than that of $H$) if $t_1\geqslant 0$. $H^T_2$ is an empty graph if $t_3=-1$ and is a smaller catacondensed hexagonal system than $H$ otherwise. $H^T_3$ is an empty graph if $t_3=-1$ or $T_3$ ends at a terminal hexagon, is a smaller catacondensed hexagonal system than $H$ if $t_3\geqslant 0$ and $T_3$ ends at a kink, and is a union of two smaller catacondensed hexagonal systems than $H$ if $t_3\geqslant 0$ and $T_3$ ends at a branched hexagon. If $h$ is an end branched hexagon, then denote $u_1,u_2$ the two edges that adjacent to the last vertical edge of $T_3$. Denote $H^T_4$ the subgraph of $H^T_3$ by deleting all the hexagons in the maximal linear hexagonal chains starting from the hexagon containing $u_1$ and the hexagon containing $u_2$ respectively (see Fig. \ref{cata-defi}(d)). Note that $H^T_4$ may be an empty graph, or a union of one, two, three, or four smaller catacondensed hexagonal systems than $H$.

    \begin{thm}
        \label{tail-anti-forc}
        The anti-forcing polynomial of a catacondensed hexagonal system $H$ possessing a tail $T(t_1,t_2,t_3)$ with respect to $h$ has the following form$:$

        \noindent$(1)$ if $t_2=-1,$ then
        \begin{align}
            \label{rela-tail2}
            Af(H,x)=&xAf(H^T_2,x)+xAf(H^T_3,x)+\alpha_T;
        \end{align}
        $(2)$ if $t_2\geqslant 0,$ then
        \begin{align}
            \label{rela-tail3}
            Af(H,x)=x^2Af(H^T_1,x)+(3x^2+2(t_1+t_2)x^3+t_1t_2x^4)Af(H^T_2,x)-x^3Af(H^T_3,x)+\beta_T,
        \end{align}
        where $\alpha_T=0$ if $t_1=-1,$ and $\alpha_T=xAf(H^T_1,x)+t_1x^2Af(H^T_2,x)-x^2Af(H^T_3,x)$ if $t_1\geqslant 0;$ and $\beta_T=xAf(H^T_2,x)-x^2Af(H^T_4,x)$ if $t_3=0,$ and $\beta_T=x^2Af(H^T_3,x)$ if $t_3\geqslant 1.$
    \end{thm}
    \begin{proof}
        \textbf{(1)} If $H$ contains a tail $T(t_1,-1,t_3)$ ($t_1,t_3\geqslant -1$, and if $t_1=-1$ then $t_3=-1$), then we in turn denote the hexagon in $T_1$ by $C_i$ for $i=1,2,\ldots,t_1+2$ such that $C_{1}=h$. Furthermore, denote the vertical edge that $C_{i}$ and $C_{i+1}$ share in common by $e_i$, and along the clockwise direction denote the rest edges of $C_{i}$ by $l_{i},r_i,e_{i-1},m_i,n_i$ for $i=1,2,\ldots,t_1+2$ (see Figs. \ref{cata-defi}(a,b)). By Lemma \ref{ding} we can divide $\mathcal{M}(H)$ in $t_1+3$ subsets:
        \begin{align*}
            \mathcal{M}_i(H)=\{M\in \mathcal{M}(H): e_i\in M\}
        \end{align*}
        for $i=0,1,\ldots,t_1+2$. By Eq. (\ref{equ1}), we have
        \begin{align}
            \label{tail2}
            Af(H,x)=&\sum_{i=0}^{t_1+2}\sum_{ M\in {\mathcal{M}_i(H)}}{ x^{ af(H,M) } }:=\sum_{i=0}^{t_1+2}Af_i(H,x).
        \end{align}

        Given $M\in \mathcal{M}_{0}(H)$. Then $r_1\notin M$. On the one hand $r_{1}$ belongs to $M$-alternating hexagon $C_{1}$. And on the other hand $r_{1}$ forces edges $l_{1},n_{1},e_0$ and anti-forces edges $e_{1},m_{1}$. By Lemma \ref{anti-forc-subs-calc} we know
        \begin{align*}
           af(H,M)=af(H\circleddash r_{1},M\cap E(H\circleddash r_{1}))+1.
        \end{align*}
        In Fig. \ref{tail2-proof}(a), we illustrate the three cases of $H\circleddash r_{1}$, where double line denotes $e_{0}$, cross denotes $r_1$, and bold lines denote the edges that forced by $r_1$. Note that in the last case, the last vertical edge $v$ of $T_3$ is a cut edge whose removal from the subgraph $H-r_1$ makes the two components to be even, which implies that $v$ is anti-forced by $r_1$. Hence we know $H\circleddash r_{1}=H^T_3$. It follows that
        \begin{align}
            \label{tail24}
            Af_{0}(H,x)=&\sum _{ M\in {\mathcal{M}_{0}(H)}}{ x^{ af(H^T_3,M\cap E(H^T_3))+1}}=\sum _{ M\in {\mathcal{M}(H^T_3)}}{x^{af(H^T_3,M)}}\cdot x=xAf(H^T_3,x).
        \end{align}

        \begin{figure}[htbp]
            \centering
            \includegraphics[height=2.8in]{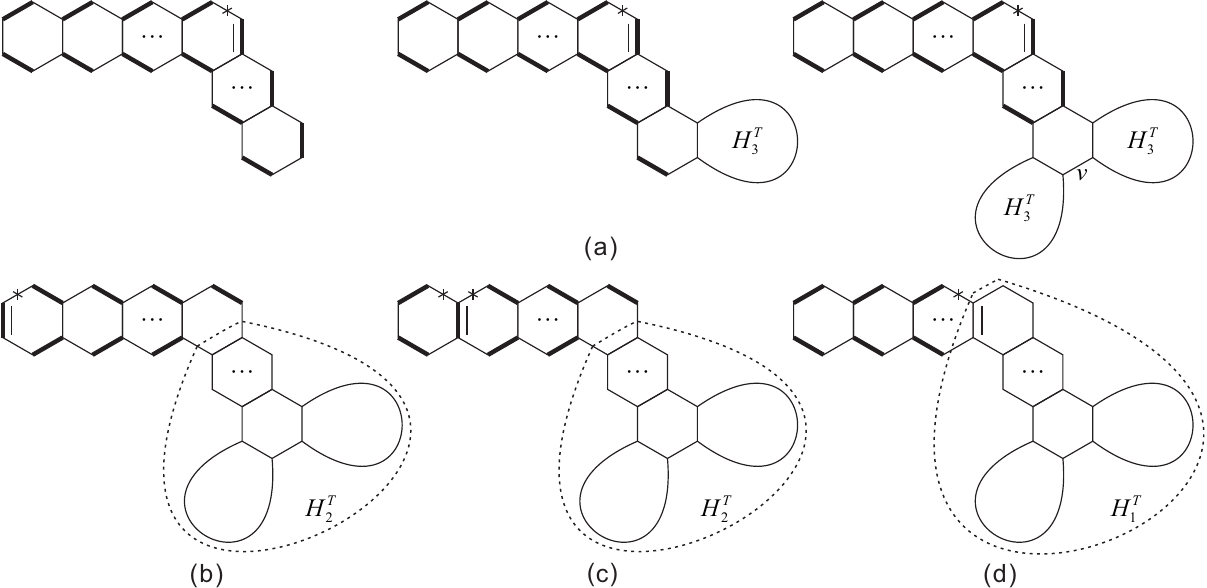}
            \caption{(a) $H\circleddash r_{1}$, (b) $H\circleddash l_{t_1+2}$, (c) $H\circleddash \{r_{i+1},l_i\}$, and (d) $H\circleddash r_{2}$.}
            \label{tail2-proof}
        \end{figure}

        Given $M\in \mathcal{M}_{t_1+2}(H)$. Then $l_{t_1+2}\notin M$. On the one hand $l_{t_1+2}$ belongs to $M$-alternating hexagon $C_{t_1+2}$. And on the other hand $l_{t_1+2}$ forces edges $e_{t_1+2},r_{t_1+2},m_{t_1+2}$ and anti-forces edges $e_{t_1+1},n_{t_1+2}$. By Lemma \ref{anti-forc-subs-calc} and Fig. \ref{tail2-proof}(b) we know
        \begin{align*}
           af(H,M)=af(H\circleddash l_{t_1+2},M\cap E(H\circleddash l_{t_1+2}))+1=af(H^T_2,M\cap E(H^T_2))+1.
        \end{align*}
        It follows that
        \begin{align}
            \label{tail21}
            Af_{t_1+2}(H,x)=&\sum _{ M\in {\mathcal{M}_{t_1+2}(H)}}{ x^{ af(H^T_2,M\cap E(H^T_2))+1}}=\sum _{ M\in {\mathcal{M}(H^T_2)}}{x^{af(H^T_2,M)}}\cdot x=xAf(H^T_2,x).
        \end{align}
        Substituting Eqs. (\ref{tail24},\ref{tail21}) into Eq. (\ref{tail2}), we immediately obtain Eq. (\ref{rela-tail2}) for $t_1=-1$.

        For $t_1\geqslant 1$, given $M\in \mathcal{M}_i(H)$ for $i=2,3\ldots,t_1+1$. Then $r_{i+1},l_{i}\notin M$. On the one hand $r_{i+1}$ and $l_{i}$ belong to compatible $M$-alternating hexagons $C_{i+1}$ and $C_{i}$ respectively. And on the other hand $\{r_{i+1},l_i\}$ forces edges $l_{i+1},n_{i+1},e_i,r_i,m_i$ and anti-forces edges $e_{i+1},m_{i+1},n_i,e_{i-1}$. By Lemma \ref{anti-forc-subs-calc} and Fig. \ref{tail2-proof}(c) we know
        \begin{align*}
           af(H,M)=af(H\circleddash \{r_{i+1},l_i\},M\cap E(H\circleddash \{r_{i+1},l_i\}))+2=af(H^T_2,M\cap E(H^T_2))+2.
        \end{align*}
        It follows that
        \begin{align}
            \label{tail22}
            Af_i(H,x)=&\sum _{ M\in {\mathcal{M}_i(H)}}{ x^{ af(H^T_2,M\cap E(H^T_2))+2}}=\sum _{ M\in {\mathcal{M}(H^T_2)}}{x^{af(H^T_2,M)}}\cdot x^2=x^2Af(H^T_2,x).
        \end{align}

        For $t_1\geqslant 0$, given $M\in \mathcal{M}_{1}(H)$. Then $r_{2}\notin M$. On the one hand $r_{2}$ belongs to $M$-alternating hexagon $C_{2}$. And on the other hand $r_{2}$ forces edges $l_{2},n_{2}$ and anti-forces edges $e_{2},m_{2}$. By Lemma \ref{anti-forc-subs-calc} and Fig. \ref{tail2-proof}(d) we know
        \begin{align*}
           af(H,M)=af(H\circleddash r_{2},M\cap E(H\circleddash r_{2}))+1=af(H^T_1,M\cap E(H^T_1))+1.
        \end{align*}
        It follows that
        \begin{align*}
            Af_{1}(H,x)=&\sum _{ M\in {\mathcal{M}_{1}(H)}}{ x^{ af(H^T_1,M\cap E(H^T_1))+1}}=\sum _{M\in {\mathcal{M}(H^T_1)},~e_{1}\in M}{x^{af(H^T_1,M)}}\cdot x \\
            :=&xAf^\ast(H^T_1,x)=\sum _{M\in {\mathcal{M}(H^T_1)}}{x^{af(H^T_1,M)}}\cdot x-\sum _{M\in {\mathcal{M}(H^T_1)},~e_1\notin M}{x^{af(H^T_1,M)}}\cdot x.
        \end{align*}
        In $H^T_1$, given $M\in {\mathcal{M}(H^T_1)}$ with $e_1\notin M$. Then $e_0\in M$. By a similar argument to the calculation of $Af_0(H,x)$ we know
        \begin{align*}
           af(H^T_1,M)=af(H^T_1 \circleddash r_{1},M\cap E(H^T_1\circleddash r_{1}))+1=af(H^T_3,M\cap E(H^T_3))+1.
        \end{align*}
        It follows that
        \begin{align}
            \label{tail23}
            Af_{1}(H,x)=&xAf(H^T_1,x)-\sum _{M\in {\mathcal{M}(H^T_1)},~e_1\notin M}{x^{af(H^T_3,M\cap E(H^T_3))+1}}\cdot x \nonumber\\
            =&xAf(H^T_1,x)-\sum _{M\in {\mathcal{M}(H^T_3)}}{x^{af(H^T_3,M)}}\cdot x^2 =xAf(H^T_1,x)-x^2Af(H^T_3,x).
        \end{align}
        Substituting Eqs. (\ref{tail24}-\ref{tail23}) into Eq. (\ref{tail2}), we immediately obtain Eq. (\ref{rela-tail2}) for $t_1\geqslant 0$.

        \textbf{(2)} If $H$ contains a tail $T(t_1,t_2,t_3)$ ($t_1,t_2,t_3\geqslant 0$), then we in turn denote the hexagon in $T_1$ by $B_i$ for $i=1,2,\ldots,t_1+2$ and the hexagon in $T_2$ by $D_j$ for $j=1,2,\ldots,t_2+2$ such that $B_{1}=D_{1}=h$. Furthermore, in turn denote the vertical edge that $B_{i}$ and $B_{i+1}$ (resp. $D_{j}$ and $D_{j+1}$) share in common by $e_i$ (resp. $e'_j$), and along the clockwise direction denote the rest edges of $B_{i}$ (resp. $D_{j}$) by $l_{i},r_i,e_{i-1},m_i,n_i$ for $i=1,2,\ldots,t_1+2$ (resp. $l'_{j},r'_j,e'_{j-1},m'_j,n'_j$ for $j=1,2,\ldots,t_2+2$) (see Fig. \ref{cata-defi}(c)). By Lemma \ref{ding} we can divide $\mathcal{M}(H)$ in $(t_1+2)(t_2+2)+1$ subsets:
        \begin{align*}
            \mathcal{M}_{i,j}(H)=\{M\in \mathcal{M}(H): e_i,e'_j\in M\}
        \end{align*}
        for $i=0,1,\ldots,t_1+2$ and $j=0,1,\ldots,t_2+2$, and $i=0$ if and only if $j=0$. By Eq. (\ref{equ1}), we have
        \begin{align}
            \label{tail3}
            Af(H,x)=&\sum_{ M\in {\mathcal{M}_{0,0}(H)}}{ x^{ af(H,M) } }+\sum_{i=1}^{t_1+2}\sum_{j=1}^{t_2+2}\sum_{ M\in {\mathcal{M}_{i,j}(H)}}{ x^{ af(H,M) } }\nonumber\\
            :=&Af_{0,0}(H,x)+\sum_{i=1}^{t_1+2}\sum_{j=1}^{t_2+2}Af_{i,j}(H,x).
        \end{align}

        Given $M\in \mathcal{M}_{t_1+2,t_2+2}(H)$. Then $l_{t_1+2},l'_{t_2+2}\notin M$. On the one hand $l_{t_1+2}$ and $l'_{t_2+2}$ belong to compatible $M$-alternating hexagons $B_{t_1+2}$ and $D_{t_2+2}$ respectively. And on the other hand $\{l_{t_1+2},l'_{t_2+2}\}$ forces edges $e_{t_1+2},r_{t_1+2},m_{t_1+2},e'_{t_2+2},r'_{t_2+2},m'_{t_2+2}$ and anti-forces edges $e_{t_1+1},n_{t_1+2},e'_{t_2+1},n'_{t_2+2}$. By Lemma \ref{anti-forc-subs-calc} and Fig. \ref{tail3-proof1}(a) we know
        \begin{align*}
           af(H,M)=&af(H\circleddash \{l_{t_1+2},l'_{t_2+2}\},M\cap E(H\circleddash \{l_{t_1+2},l'_{t_2+2}\}))+2\\
           =&af(H^T_2,M\cap E(H^T_2))+2.
        \end{align*}
        It follows that
        \begin{align}
            \label{tail31}
            Af_{t_1+2,t_2+2}(H,x)=&\sum _{ M\in {\mathcal{M}_{t_1+2,t_2+2}(H)}}{ x^{ af(H^T_2,M\cap E(H^T_2))+2}}=\sum _{ M\in {\mathcal{M}(H^T_2)}}{x^{af(H^T_2,M)}}\cdot x^2\nonumber\\
            =&x^2Af(H^T_2,x).
        \end{align}

        \begin{figure}[htbp]
            \centering
            \includegraphics[height=2in]{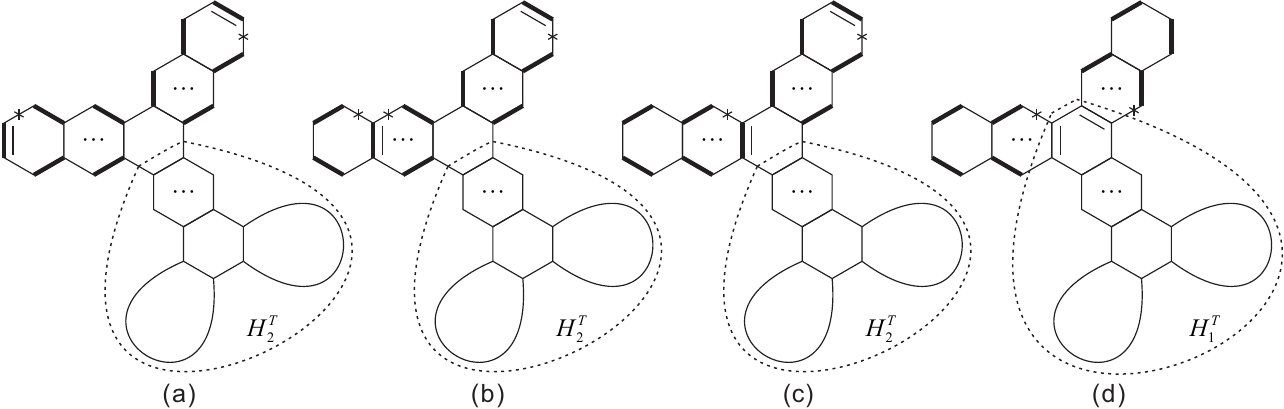}
            \caption{(a) $H\circleddash \{l_{t_1+2},l'_{t_2+2}\}$, (b) $H\circleddash \{r_{i+1},l_{i},l'_{t_2+2}\}$, (c) $H\circleddash \{r_{2},l'_{t_2+2}\}$, and (d) $H\circleddash \{r_{2},r'_{2}\}$.}
            \label{tail3-proof1}
        \end{figure}

        Given $M\in \mathcal{M}_{i,t_2+2}(H)$ for $i=2,3\ldots,t_1+1$. Then $r_{i+1},l_{i},l'_{t_2+2}\notin M$. On the one hand $r_{i+1}$, $l_{i}$ and $l'_{t_2+2}$ belong to compatible $M$-alternating hexagons $B_{i+1}$, $B_{i}$ and $D_{t_2+2}$ respectively. And on the other hand $\{r_{i+1},l_{i},l'_{t_2+2}\}$ forces edges $l_{i+1},n_{i+1},e_i,r_i,m_i,e'_{t_2+2},r'_{t_2+2},$\\$m'_{t_2+2}$ and anti-forces edges $e_{i+1},m_{i+1},n_i,e_{i-1},e'_{t_2+1},n'_{t_2+2}$. By Lemma \ref{anti-forc-subs-calc} and Fig. \ref{tail3-proof1}(b) we know
        \begin{align*}
           af(H,M)=&af(H\circleddash \{r_{i+1},l_{i},l'_{t_2+2}\},M\cap E(H\circleddash \{r_{i+1},l_{i},l'_{t_2+2}\}))+3\\
           =&af(H^T_2,M\cap E(H^T_2))+3.
        \end{align*}
        It follows that
        \begin{align}
            \label{tail32}
            Af_{i,t_2+2}(H,x)=&\sum _{ M\in {\mathcal{M}_{i,t_2+2}(H)}}{ x^{ af(H^T_2,M\cap E(H^T_2))+3}}=\sum _{ M\in {\mathcal{M}(H^T_2)}}{x^{af(H^T_2,M)}}\cdot x^3=x^3Af(H^T_2,x).
        \end{align}

        Given $M\in \mathcal{M}_{1,t_2+2}(H)$. Then $r_{2},l'_{t_2+2}\notin M$. On the one hand $r_{2}$ and $l'_{t_2+2}$ belong to compatible $M$-alternating hexagons $B_{2}$ and $D_{t_2+2}$ respectively. On the other hand $\{r_{2},l'_{t_2+2}\}$ forces edges $l_{2},n_{2},e_1,e'_{t_2+2},r'_{t_2+2},m'_{t_2+2}$ and anti-forces edges $e_{2},m_{2},e'_{t_2+1},n'_{t_2+2}$. By Lemma \ref{anti-forc-subs-calc} and Fig. \ref{tail3-proof1}(c) we know
        \begin{align*}
           af(H,M)=af(H\circleddash \{r_{2},l'_{t_2+2}\},M\cap E(H\circleddash\{r_{2},l'_{t_2+2}\}))+2=af(H^T_2,M\cap E(H^T_2))+2.
        \end{align*}
        It follows that
        \begin{align}
            \label{tail33}
            Af_{1,t_2+2}(H,x)=&\sum _{ M\in {\mathcal{M}_{1,t_2+2}(H)}}{ x^{ af(H^T_2,M\cap E(H^T_2))+2}}=\sum _{ M\in {\mathcal{M}(H^T_2)}}{x^{af(H^T_2,M)}}\cdot x^2\nonumber\\
            =&x^2Af(H^T_2,x).
        \end{align}

        By a similar argument to above, we know for $i=2,3,\ldots,t_1+1$ and $j=2,3,\ldots,t_2+1$
        \begin{align}
            \label{tail34}
            Af_{t_1+2,j}(H,x)&=Af_{1,j}(H,x)= Af_{i,1}(H,x)=x^3Af(H^T_2,x),\\
            Af_{t_1+2,1}(H,x)&=x^2Af(H^T_2,x),\\
            Af_{i,j}(H,x)&=x^4Af(H^T_2,x).
        \end{align}

        Given $M\in \mathcal{M}_{1,1}(H)$. Then $r_{2},r'_{2}\notin M$. On the one hand $r_{2}$ and $r'_{2}$ belong to compatible $M$-alternating hexagons $B_{2}$ and $D_2$ respectively. On the other hand $\{r_{2},r'_{2}\}$ forces edges $l_{2},n_{2},l'_{2},n'_{2}$ and anti-forces edges $e_{2},m_{2},e'_{2},m'_{2}$. By Lemma \ref{anti-forc-subs-calc} and Fig. \ref{tail2-proof}(d) we know
        \begin{align*}
           af(H,M)=af(H\circleddash \{r_{2},r'_{2}\},M\cap E(H\circleddash \{r_{2},r'_{2}\}))+2=af(H^T_1,M\cap E(H^T_1))+2.
        \end{align*}
        By a similar argument to the calculation of Eq. (\ref{tail23}) we have
        \begin{align}
            \label{tail35}
            Af_{1,1}(H,x)=&\sum _{ M\in {\mathcal{M}_{1,1}(H)}}{ x^{ af(H^T_1,M\cap E(H^T_1))+2}}=\sum _{M\in {\mathcal{M}(H^T_1)},~e_1,e'_{1}\in M}{x^{af(H^T_1,M)}}\cdot x^2 \nonumber \\
            =&x^2Af^{\ast}(H^T_1,x)=x^2Af(H^T_1,x)-x^3Af(H^T_3,x).
        \end{align}

        It remains to consider $Af_{0,0}(H,x)$. Denote $s$ the edge adjacent to $n_1$ and $m_1$, $t$ the edge adjacent to $e_0$ and $m_1$. Obviously there is a bijection $g$ between $\mathcal{M}_{0,0}(H)$ (see Fig. \ref{tail3-proof2}(a)) and $\{M\in \mathcal{M}(H^T_2):m_1\in M\}$ (see Fig. \ref{tail3-proof2}(b)) through replacing the $M$-alternating path $n_1m_2n_2\cdots m_{t_1+2}n_{t_1+2}e_{t_1+2}l_{t_1+2}r_{t_1+2}\cdots l_2r_2l_1m'_2n'_2\cdots m'_{t_2+2}n'_{t_2+2}e'_{t_2+2}l'_{t_2+2}r'_{t_2+2}\cdots l'_2r'_2e_0$\\by $g(M)$-matched edge $m_1$. Given $M\in \mathcal{M}_{0,0}(H)$ and a minimum anti-forcing set $S$ of $g(M)$ in $H^T_2$. Then $m_1\notin S$. We claim that $S\cup \{m_1\}$ is an anti-forcing set of $M$. Obviously every $M$-alternating cycle in $T_1\cup T_2$ must contain $m_1$, and every $M$-alternating cycle in $H^T_2$ must contain some edge in $S$. If there is no other $M$-alternating cycles in $H$, then the claim holds. Otherwise, denote a such one by $X$. Then $X$ must contain $M$-alternating paths $n_1s$ and $e_0t$. Since $X$ can be transformed into a $g(M)$-alternating cycle in $H^T_2$ through replacing the $M$-alternating path in $T_1\cup T_2$ starting from $n_1$ and ending at $e_0$ by $g(M)$-matched edge $m_1$, we have that $X$ must contain some edge in $S$. Then the claim holds in this case, which implies $af(H,M)\leqslant af(H^T_2,g(M))+1$.

        \begin{figure}[htbp]
            \centering
            \includegraphics[height=2in]{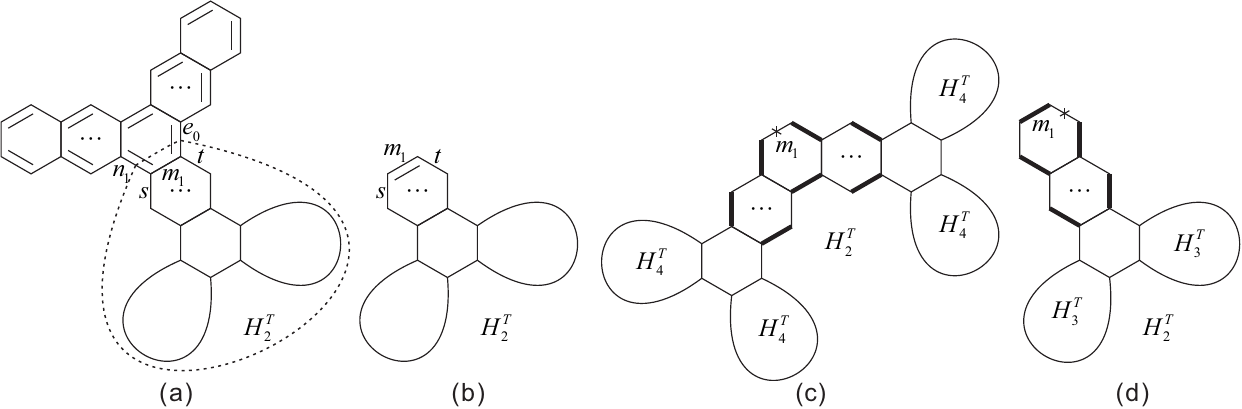}
            \caption{Illustration of calculation for $Af_{0,0}(H,x)$.}
            \label{tail3-proof2}
        \end{figure}

        Given a maximum compatible $g(M)$-alternating set $\mathcal{C}$ of $H^T_2$. We claim that $H$ contains a compatible $M$-alternating set with cardinality $|\mathcal{C}|+1$. Obviously every $g(M)$-alternating cycle in $\mathcal{C}$ that does not contain $m_1$ is also an $M$-alternating cycle of $H$, and the $M$-alternating cycle $h$ is compatible with the above every such one in $H$. If $\mathcal{C}$ does not contain $m_1$, then $\mathcal{C}\cup \{h\}$ is a compatible $M$-alternating set of $H$ and the claim holds. Otherwise, such one is unique since it contains $s$, hence we denote it by $Y$. Since $M$-alternating cycle $Z$ in $H$ through replacing $m_1$ by $n_1m_2n_2e_2l_2r_2l_1m'_2n'_2e'_2l'_2r'_2e_0$ is compatible with every member in $\mathcal{C}\setminus \{Y\}\cup \{h\}$, we have that $\mathcal{C}\setminus \{Y\}\cup \{h,Z\}$ is a compatible $M$-alternating set in $H$. Then the claim holds in this case, which implies $af(H,M)\geqslant af(H^T_2,g(M))+1$. It follows that $af(H,M)= af(H^T_2,g(M))+1$ and
        \begin{align*}
            Af_{0,0}(H,x)=\sum_{M\in\mathcal{M}_{0,0}(H)}{x^{af(H^T_2,g(M))+1}}=\sum_{M\in {\mathcal{M}(H^T_2)},~m_1\in M}{x^{af(H^T_2,M)}}\cdot x:=xAf^{\star}(H^T_2,x).
        \end{align*}

        If $t_3=0$, then by a similar argument to the calculation of Eq. (\ref{tail23}) and Fig. \ref{tail3-proof2}(c), we have
        \begin{align}
            \label{tail36}
            Af_{0,0}(H,x)=&\sum_{M\in \mathcal{M}(H^T_2)}{x^{af(H^T_2,M)}}\cdot x-\sum_{M\in {\mathcal{M}(H^T_2)},~m_1\notin M}{x^{af(H^T_2,M)}}\cdot x\nonumber\\
            =&xAf(H^T_2,x)-x^2Af(H^T_4,x).
        \end{align}
        If $t_3\geqslant 1$, then by a similar argument to the above and Fig. \ref{tail3-proof2}(d), we have
        \begin{align}
            \label{tail37}
            Af_{0,0}(H,x)=x^2Af(H^T_3,x).
        \end{align}
        Substituting Eqs. (\ref{tail31}-\ref{tail37}) into Eq. (\ref{tail3}), we immediately obtain Eq. (\ref{rela-tail3}).
    \end{proof}

    Deng and Zhang \cite{deng2} confirmed the continuity of anti-forcing spectrum for catacondensed hexagonal systems by analysing properties of graphs directly, and here we show the result by degrees of anti-forcing polynomial.

    \begin{cor}{\em\cite{deng2}}
        The anti-forcing spectrum of a catacondensed hexagonal system $H$ is continuous$,$ namely it is an integer interval$.$
    \end{cor}
    \begin{proof}
        We proceed it by induction on the number of hexagons $m$ of $H$. If $m=1$, then $H=T(-1,-1,-1)$ and has anti-forcing polynomial $2x$. Thus the anti-forcing spectrum is $\{1\}$. Suppose that the result holds for the cases of less than $m(\geqslant 2)$. Now we consider the case of $m$. Pick a tail $T(t_1,t_2,t_3)$ of $H$, where $t_1\geqslant 0$. By assumption, the anti-forcing spectrum of $H^T_k$ is continuous, say integer interval $[i_k,j_k]$ for $k=1,2,3,4$.

        Obviously there is a bijection $g$ between $\mathcal{M}(H^T_3)$ and $\{M\in \mathcal{M}(H^T_2):m_1\in M\}$, where $g(M)$ is obtained from $M$ by adding some $g(M)$-matched edges in $T_3$ such that $m_1$ is the vertical edge of $T_3$ contained in $g(M)$. Since every compatible $M$-alternating set of $H^T_3$ is also a compatible $g(M)$-alternating set of $H^T_2$, and the union of $\{s\}$ and an anti-forcing set of $M$ is an anti-forcing set of $g(M)$, we have that $af(H^T_3,M)\leqslant af(H^T_2,g(M))\leqslant af(H^T_3,M)+1$. Then the set of degrees of $Af^{\star}(H^T_2,x)$ is contained in $[i_3,j_3+1]$, and the set $[i_2,j_2]\cup [i_3,j_3]$ of degrees of $Af(H^T_2,x)+Af(H^T_3,x)$ is continuous.

        \textbf{Case 1.} $t_2=-1.$ Denote $\hat{H}$ the subgraph of $H$ by deleting $C_{t_1+2}$. Hence $\hat{H}$ is a catacondensed hexagonal system with $m-1$ hexagons and suppose its anti-forcing spectrum is an integer interval $[a,b]$.

        \textbf{Subcase 1.1.} $t_1\geqslant 1$. Pick the tail $T(t_1-1,-1,t_3)$ of $\hat{H}$ obtained from $T(t_1,-1,t_3)$ by deleting $C_{t_1+2}$. By the calculation of Eq. (\ref{rela-tail2}) we have
        \begin{align*}
            Af(H,x)=&(x+t_1x^2)Af(H^T_2,x)+xAf(H^T_3,x)+xAf^\ast(H^T_1,x),\\
            Af(\hat{H},x)=&(x+(t_1-1)x^2)Af(H^T_2,x)+xAf(H^T_3,x)+xAf^\ast(H^T_1,x)\\
            =&Af(H,x)-x^2Af(H^T_2,x).
        \end{align*}
        Obviously the degrees of $xAf(H^T_2,x)$ and degrees of $x^2Af(H^T_2,x)$ form $[i_2+1,j_2+1]$ and $[i_2+2,j_2+2]$, respectively. Since the existence of $xAf(H^T_2,x)$ in $Af(\hat{H},x)$, we know $a\leqslant i_2+1\leqslant j_2+1 \leqslant b$. If follows that the anti-forcing spectrum of $H$ is $[a,b+1]$ if $j_2+1=b$, and $[a,b]$ if $j_2+1<b$.

        \textbf{Subcase 1.2.} $t_1=0$. Then $\hat{H}=H^T_1$ and by a similar argument to the calculation of Eq. (\ref{rela-tail2}) we have
        \begin{align*}
            Af(H,x)=&xAf(H^T_2,x)+xAf(H^T_3,x)+xAf^\ast(H^T_1,x),\\
            Af(\hat{H},x)=&xAf(H^T_3,x)+Af^\ast(H^T_1,x).
        \end{align*}
        By assumption, we know that the union of $[i_3+1,j_3+1]$ and the set of degrees of $Af^\ast(H^T_1,x)$ form $[a,b]$. We claim that either $xAf(H^T_2,x)$ or $xAf^\ast(H^T_1,x)$ has degree $j_3+2$. If the claim holds, then the anti-forcing spectrum of $H$ is continuous since the union of set of degrees of $xAf(H^T_2,x)+xAf(H^T_3,x)$ and $\{j_3+2\}$ is an integer interval $[i_2+1,j_2+1]\cup [i_3+1,j_3+2]$.

        Given a perfect matching $M$ such that $af(H^T_3,M)=j_3$. If $af(H^T_2,g(M))=af(H^T_3,$\\$M)+1$, then $xAf(H^T_2,x)$ has degree $j_3+2$. If $af(H^T_2,g(M))=af(H^T_3,M)$, then $af(H^T_3,M)+2=af(H,d(M))$, where $d(M)\in \mathcal{M}_1(H)$ is obtained from $g(M)$ by adding $l_2,n_2,e_1,r_1$. This implies that $xAf^\ast(H^T_1,x)$ has degree $j_3+2$. Then the claim holds.

        \textbf{Case 2.} $t_2\geqslant 0$.

        \textbf{Subcase 2.1.} $t_1\geqslant 1$. Denote $\tilde{H}$ the subgraph of $H$ by deleting $B_{t_1+2}$. Hence $\tilde{H}$ is a catacondensed hexagonal system with $m-1$ hexagons. Pick the tail $T(t_1-1,t_2,t_3)$ of $\tilde{H}$ obtained from $T(t_1,t_2,t_3)$ by deleting $B_{t_1+2}$. By the calculation of Eq. (\ref{rela-tail3}) we have
        \begin{align*}
            Af(H,x)=&(3x^2+2(t_1+t_2)x^3+t_1t_2x^4)Af(H^T_2,x)+x^2Af^{\ast}(H^T_1,x)+xAf^{\star}(H^T_2,x),\\
            Af(\tilde{H},x)=&(3x^2+2(t_1+t_2-1)x^3+(t_1-1)t_2x^4)Af(H^T_2,x)+x^2Af^{\ast}(H^T_1,x)\\
            &+xAf^{\star}(H^T_2,x)\\
            =&Af(H,x)-(2x^3+t_2x^4)Af(H^T_2,x).
        \end{align*}
        By a similar argument to Subcase 1.1, if $t_2=0$, then the anti-forcing spectrum of $H$ is continuous since the existence of $3x^2Af(H^T_2,x)$ in $Af(\tilde{H},x)$; and if $t_2\geqslant 1$, then the anti-forcing spectrum of $H$ is continuous since the existence of $2(t_1+t_2-1)x^3Af(H^T_2,x)$ in $Af(\tilde{H},x)$.

        \textbf{Subcase 2.2.} $t_1=0$ and $t_2\geqslant 1$. Denote $\tilde{H}$ the subgraph of $H$ by deleting $D_{t_2+2}$. By a similar to Subcase 2.1, we know that the anti-forcing spectrum of $H$ is continuous.

        \textbf{Subcase 2.3.} $t_1=t_2=0$. Denote $\tilde{H}$ the subgraph of $H$ by deleting $D_{2}$. Hence $\tilde{H}$ is a catacondensed hexagonal system with $m-1$ hexagons and suppose its anti-forcing spectrum is $[p,q]$. Pick the tail $T(0,-1,t_3)$ of $\tilde{H}$ obtained from $T(0,0,t_3)$ by deleting $D_{2}$. By the calculation of Eqs. (\ref{rela-tail2},\ref{rela-tail3}) we have
        \begin{align*}
            Af(H,x)=&3x^2Af(H^T_2,x)+x^2Af^\ast(H^T_1,x)+xAf^\star(H^T_2,x),\\
            Af(\tilde{H},x)=&xAf(H^T_2,x)+xAf^\ast(H^T_1,x)+xAf(H^T_3,x).
        \end{align*}
        By assumption, we know that the union of the set of degrees of $xAf(H^T_2,x)+xAf^\ast(H^T_1,x)$ and $[i_3+1,j_3+1]$ form $[p,q]$. Note that the set of degrees of $xAf^\star(H^T_2,x)$ is contained in $[i_3+1,j_3+2]$. We claim that either $xAf^\star(H^T_2,x)$ or $3x^2Af(H^T_2,x)$ has degree $w+2$ for every $w\in [i_3,j_3]$. If the claim holds, then the anti-forcing spectrum of $H$ is continuous since the union of set of degrees of $3x^2Af(H^T_2,x)+x^2Af^\ast(H^T_1,x)$ and $[i_3+2,j_3+2]$ is an integer interval $[p+1,q+1]$.

        Given a perfect matching $M$ such that $af(H^T_3,M)=w$. If $af(H^T_2,g(M))=af(H^T_3,$\\$M)+1$, then $xAf^\star(H^T_2,x)$ has degree $w+2$. If $af(H^T_2,g(M))=af(H^T_3,M)$, then $3x^2Af(H^T_2,x)$ has degree $w+2$. Then the claim holds.
    \end{proof}

    We now give some anti-forcing polynomials of particular catacondensed hexagonal systems. Given a hexagonal chain with $n(\geqslant 1)$ maximal linear hexagonal chains and in which containing $r_1+2,r_2+2,\ldots,r_n+2$ ($r_i\geqslant 0$ for $i=1,2,\ldots,n$) hexagons in turn (see Fig. \ref{cata-eg}(b)). For convenience, we denote it by $Hl(r_1,r_2,\ldots,r_n)$, and make a convention that $Hl(r_1,r_2,\ldots,r_i,-1)=Hl(r_1,r_2,\ldots,r_i)$ for $i\geqslant 0$, $Hl(r_1,\ldots,r_0)$ is a hexagon, and $Hl(r_1,\ldots,r_{-1}-1)=Hl(r_1,\ldots,r_0-1)$ is an empty graph. Pick the tail $T(r_n,-1,t_3)$, where $t_3=-1$ if $n=1$ and $t_3=r_{n-1}$ if $n\geqslant 2$. Then $H^T_1=Hl(r_1,r_2,\ldots,r_{n-1})$, $H^T_2=Hl(r_1,r_2,\ldots,r_{n-1}-1)$, and $H^T_3=Hl(r_1,r_2,\ldots,r_{n-2}-1)$. Then we have the following conclusion.

    \begin{cor}{\em\cite{lei2}}
        \label{linear}
        The anti-forcing polynomial of a hexagonal chain $Hl(r_1,r_2,\ldots,r_n)$ with at least two hexagons is
        \begin{align*}
            Af(Hl(r_1,\ldots,r_n),x)=&(x+r_{n}x^2)Af(Hl(r_1,\ldots,r_{n-1}-1),x)\\
            +&(x-x^2)Af(Hl(r_1,\ldots,r_{n-2}-1),x)+xAf(Hl(r_1,\ldots,r_{n-1}),x),
        \end{align*}
        where $Af(Hl(r_1,r_2,\ldots,r_i,-1),x)=Af(Hl(r_1,r_2,\ldots,r_i),x)$ for $i\geqslant 0,$ $Af(Hl(r_1,\ldots,$\\$r_0),x)=2x,$ and $Af(Hl(r_1,\ldots,r_{-1}-1),x)=Af(Hl(r_1,\ldots,r_{0}-1),x)=1.$
    \end{cor}

    Denote the anti-forcing polynomials of graphs illustrated in Figs. \ref{cata-eg3}(a-c) by $P_n,Q_n,R_n$ ($n\geqslant 1$) , where the third graph consists of $n$ rows of maximal linear hexagonal chains of length two, the second graph is obtained from the third one by adding a hexagon above the first row or below the last row, and the first graph is obtained from the third one by adding a hexagon above the first row and adding a hexagon below the last row.

    \begin{figure}[htbp]
        \centering
        \includegraphics[height=1.8in]{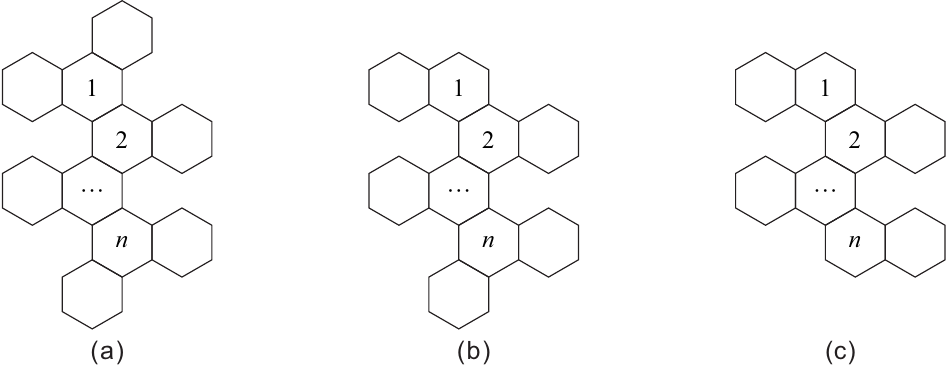}
        \caption{Examples of catacondensed hexagonal system.}
        \label{cata-eg3}
    \end{figure}

    \begin{eg}
        $P_n,Q_n,R_n$ have the following recurrence relations$:$
        \begin{align}
            \label{e1}
            P_n&=x^2P_{n-1}+(x+3x^2)Q_{n-1}-2x^4Q_{n-2}-2x^3Q_{n-3} ~~\text{ for $n\geqslant 3$},\\
            \label{e6}
            Q_n&=x^2Q_{n-1}+(x+3x^2)R_{n-1}-2x^4R_{n-2}-2x^3R_{n-3}~~ \text{ for $n\geqslant 3$},\\
            \label{e2}
            Q_n&=xP_{n-1}+xQ_{n-1}+(2x^2-2x^3)Q_{n-2} ~~\text{ for $n\geqslant 2$},\\
            \label{e3}
            R_n&=xQ_{n-1}+xR_{n-1}+(2x^2-2x^3)R_{n-2}~~\text{ for $n\geqslant 2$},
        \end{align}
        where $P_0=x^2+2x,$ $P_1=x^4+7x^3+x^2,$ $P_2=x^6+6x^5+11x^4+6x^3,$ $Q_0=2x,$ $Q_1=x^3+3x^2+x,$ $Q_2=x^5+4x^4+8x^3+x^2,$ $R_0=1,$ $R_1=x^2+2x,$ $R_2=x^4+2x^3+5x^2.$

        Furthermore$,$ for $n\geqslant 0$ the explicit expression of $R_n$ is
        \begin{align}
          R_n&=\sum_{k=0}^{n}\sum_{a=0}^{\lfloor\frac{n-k}{2}\rfloor}\sum_{c=\lceil\frac{n-k-2a}{2}\rceil}^{\lfloor\frac{n-k-a}{2}\rfloor}\binom{k+a}{a}\binom{a}{n-k-a-2c}\binom{k+c}{c}2^cx^{k+2a+2c}\nonumber\\
            &+\sum_{k=0}^{n-1}\sum_{a=0}^{\lfloor\frac{n-k-1}{2}\rfloor}\sum_{c=\lceil\frac{n-k-2a-1}{2}\rceil}^{\lfloor\frac{n-k-a-1}{2}\rfloor}\binom{k+a}{a}\binom{a}{n-k-a-2c-1}\binom{k+c}{c}2^cx^{k+2a+2c+1}\nonumber\\
            &-\sum_{k=0}^{n-2}\sum_{a=0}^{\lfloor\frac{n-k-2}{2}\rfloor}\sum_{c=\lceil\frac{n-k-2a-2}{2}\rceil}^{\lfloor\frac{n-k-a-2}{2}\rfloor}\binom{k+a}{a}\binom{a}{n-k-a-2c-2}\binom{k+c}{c}2^cx^{k+2a+2c+3}\nonumber\\
            \label{e5}
            &-\sum_{k=0}^{n-3}\sum_{a=0}^{\lfloor\frac{n-k-3}{2}\rfloor}\sum_{c=\lceil\frac{n-k-2a-3}{2}\rceil}^{\lfloor\frac{n-k-a-3}{2}\rfloor}\binom{k+a}{a}\binom{a}{n-k-a-2c-3}\binom{k+c}{c}2^cx^{k+2a+2c+3},
        \end{align}
        which can be applied to derive explicit expressions of $P_n$ and $Q_n$ by Eqs. (\ref{e2},\ref{e3})$.$

    \end{eg}
    \begin{proof}
        The recurrence relations Eqs. (\ref{e1},\ref{e6}) (resp. Eqs. (\ref{e2},\ref{e3})) can be obtained from Eq. (\ref{rela-tail3}) (resp. Eq. (\ref{rela-tail2})) directly. Furthermore from Eq. (\ref{e3}) for $n\geqslant 2$ we can derive
        \begin{align*}
            xQ_{n-1}=R_n-xR_{n-1}+(2x^3-2x^2)R_{n-2}.
        \end{align*}
        Substituting it into Eq. (\ref{e6}), for $n\geqslant 4$ we have the following recurrence relation:
        \begin{align*}
            R_n-(x+x^2)R_{n-1}-3x^2R_{n-2}+2x^4R_{n-3}+2x^4R_{n-4}=0.
        \end{align*}

        Define the generating function of $R_n$ as $W=\sum _{n=0}^{\infty}R_nt^n$. By the above equation we have
        \begin{align*}
            \sum_{n=4}^{\infty}R_nt^n-(x+x^2)\sum_{n=4}^{\infty}R_{n-1}t^n-3x^2\sum_{n=4}^{\infty}R_{n-2}t^n+2x^4\sum_{n=4}^{\infty}R_{n-3}t^n+2x^4\sum_{n=4}^{\infty}R_{n-4}t^n=0.
        \end{align*}
        Substituting $R_0=1,$ $R_1=x^2+2x,$ $R_2=x^4+2x^3+5x^2,$ and $R_3=x^6+3x^5+8x^4+10x^3$ into it, we have
        \begin{align*}
          &W-1-(x^2+2x)t-(x^4+2x^3+5x^2)t^2-(x^6+3x^5+8x^4+10x^3)t^3\\
          &-(x+x^2)t[W-1-(x^2+2x)t-(x^4+2x^3+5x^2)t^2]\\
          &-3x^2t^2[W-1-(x^2+2x)t]+2x^4t^3[W-1]+2x^4t^4W=0.
        \end{align*}
        Simplifying the above equation gives
        \begin{align*}
            W\left[1-\frac{xt}{(1-x^2t-x^2t^2)(1-2x^2t^2)}\right]=\frac{1+xt-x^3t^2-x^3t^3}{(1-x^2t-x^2t^2)(1-2x^2t^2)}.
        \end{align*}
        This yields
        \begin{align*}
          &\sum_{n=0}^{\infty}R_nt^n=\sum_{k=0}^{\infty}\frac{x^{k}t^k(1+xt-x^3t^2-x^3t^3)}{(1-x^2t-x^2t^2)^{k+1}(1-2x^2t^2)^{k+1}}\\
          =&\sum_{k=0}^{\infty}x^{k}t^k(1+xt-x^3t^2-x^3t^3)\sum_{a=0}^{\infty}\binom{k+a}{a}\sum_{b=0}^{a}\binom{a}{b}x^{2a}t^{a+b}\sum_{c=0}^{\infty}\binom{k+c}{c}2^cx^{2c}t^{2c}.
        \end{align*}
        Extracting the coefficient of $t^n$ from both sides of the above equation gives Eq. (\ref{e5}).
    \end{proof}


\begin{thebibliography}{99}
%    \bibitem{chechen}Z. Che, Z. Chen, Forcing on perfect matchings --- A survey, {\it MATCH Commun. Math. Comput. Chem.\/} {\bf 66} (2011) 93--136.
%    \bibitem{fibo}S.J. Cyvin, I. Gutman, {\it Kekul\'e Structures in Benzenoid Hydrocarbons\/}, in: Lecture Notes in Chemistry, vol. 46, Springer, Berlin, 1988.
%    \bibitem{han2}H. Deng, The anti-forcing number of double hexagonal chains, {\it MATCH Commun. Math. Comput. Chem.\/} {\bf 60} (2008) 183--192.
%    \bibitem{han1}H. Deng, The anti-forcing number of hexagonal chains, {\it MATCH Commun. Math. Comput. Chem.\/} {\bf 58} (2007) 675--682.
%    \bibitem{dengxinxin}K. Deng, S. Liu, X. Zhou, Forcing and anti-forcing polynomials of perfect matchings of a pyrene system, {\it MATCH Commun. Math. Comput. Chem.\/} {\bf 85} (2021) 27--46.
%    \bibitem{deng1}K. Deng, H. Zhang, Anti-forcing spectra of perfect matchings of graphs, {\it J. Comb. Optim.\/} {\bf 33} (2017) 660--680.
    \bibitem{deng2}K. Deng, H. Zhang, Anti-forcing spectrum of any cata-condensed hexagonal system is continuous, {\it Front. Math. China\/} {\bf 12} (2017) 19--33.
%    \bibitem{apply}K. Deng, H. Zhang, Extremal anti-forcing numbers of perfect matchings of graphs, {\it Discrete Appl. Math.\/} {\bf 224} (2017) 69--79.
%    \bibitem{5}P. Hansen, M. Zheng, Bonds fixed by fixing bonds, {\it J. Chem. Inf. Comput. Sci.\/} {\bf 34} (1994) 297--304.
%    \bibitem{zheng}P. Hansen, M. Zheng, Recursive and explicit formulae for the degree of freedom of cata--condensed benzenoid hydrocarbons, {\it MATCH Commun. Math. Comput. Chem.\/} {\bf 31} (1994) 111--122.
%    \bibitem{original}F. Harary, D.J. Klein, T.P. \v{Z}ivkovi\'c, Graphical properties of polyhexes: Perfect matching vector and forcing, {\it J. Math. Chem.\/} {\bf 6} (1991) 295--306.
    \bibitem{lei2}H.-K. Hwang, H. Lei, Y.-N. Yeh, H. Zhang, Distribution of forcing and anti-forcing numbers of random perfect matchings on hexagonal chains and crowns, preprint, 2015. http://140.109.74.92/hk/?p=873.
%    \bibitem{early}D.J. Klein, M. Randi\'c, Innate degree of freedom of a graph, {\it J. Comput. Chem.\/} {\bf 8} (1987) 516--521.
    \bibitem{new}D.J. Klein, V. Rosenfeld, Forcing, freedom, \& uniqueness in graph theory \& chemistry, {\it Croat. Chem. Acta\/} {\bf 87} (2014) 49--59.
    \bibitem{Lei}H. Lei, Y.-N. Yeh, H. Zhang, Anti-forcing numbers of perfect matchings of graphs, {\it Discrete Appl. Math.\/} {\bf 202} (2016) 95--105.
%    \bibitem{antiedge}X. Li, Hexagonal systems with forcing single edges, {\it Discrete Appl. Math.\/} {\bf 72} (1997) 295--301.
%    \bibitem{RK}M. Randi\'c, D.J. Klein, Kekule valence structures revisited. Innate degrees of freedom of pi-electron couplings, in: N. Trinajsti\'c (Ed.), {\it Mathematical and Computational Concepts in Chemistry}, John Wiley \& Sons, New York, 1985, pp. 274--282.
%    \bibitem{shiupper}L. Shi, H. Zhang, Tight upper bound on the maximum anti-forcing numbers of graphs, {\it Discrete Math. Theor. Comput. Sci. \/} {\bf 19} (2017) 1--15.
%    \bibitem{antiforcing}D. Vuki\v{c}evi\'{c}, N. Trinajsti\'c, On the anti-forcing number of benzenoids, {\it J. Math. Chem.\/} {\bf 42} (2007) 575--583.
%    \bibitem{vu2} D. Vuki\v{c}evi\'{c}, N. Trinajsti\'c, On the anti-Kekul\'e number and anti-forcing number of cata-condensed benzenoids, {\it J. Math. Chem.\/} {\bf 43} (2008) 719--726.
%    \bibitem{jiangxinxin}H. Zhang, X. Jiang, Continuous forcing spectra of even polygonal chains, {\it Acta Math. Appl. Sinica, English Series\/} {\bf 37} (2021) 337--347.
%    \bibitem{zhao}H. Zhang, S. Zhao, R. Lin, The forcing polynomial of catacondensed hexagonal systems, {\it MATCH Commun. Math. Comput. Chem.\/} {\bf 73} (2015) 473--490.
%    \bibitem{forchs}S. Zhao, Matching forcing polynomials of constructable hexagonal systems, submitted.
%    \bibitem{zhaoxin}S. Zhao, H. Zhang, Anti-forcing polynomials for benzenoid systems with forcing edges, {\it Discrete Appl. Math.\/} {\bf 250} (2018) 342--356.
%    \bibitem{forc-anti-grids}S. Zhao, H. Zhang, Forcing and anti-forcing polynomials of perfect matchings for some rectangle grids, {\it J. Math. Chem.\/} {\bf 57} (2019) 202--225.
%    \bibitem{zhao1}S. Zhao, H. Zhang, Forcing polynomials of benzenoid parallelogram and its related benzenoids, {\it Appl. Math. Comput.\/} {\bf 284} (2016) 209--218.
    \end{thebibliography}
\end{document}